\documentclass[11pt]{article}
\usepackage{amssymb}
\usepackage{srcltx}

\usepackage{amsmath,amsfonts,latexsym}
\usepackage{latexsym,amsfonts,euscript,amssymb}
\numberwithin{equation}{section}

\newtheorem{theorem}{Theorem}[section]
\newtheorem{definition}[theorem]{Definition}
\newtheorem{lemma}[theorem]{Lemma}
\newtheorem{corollary}[theorem]{Corollary}

\newenvironment{proof}{\underline{Proof:}}{ \hfill $\Box$ 
                      \vspace{\baselineskip}}

\begin{document}

\pagestyle{myheadings}
\markright{Solvability of a finite group}

\title{A sufficient conditon for solvability of finite groups}

\author{Wujie Shi
}

\date{}
\maketitle

\begin{center}
Dept. of Math., Chongqing Univ. of Arts and Science, Chongqing 402160
\end{center}

\footnotetext{The author gratefully acknowledges
  the support by
National Natural Science Foundation of China (Grant No. 11171364, 11271301, 11671063).

}
\renewcommand{\arraystretch}{0.1}

\begin{abstract}
The following theorem is proved: Let $G$ be a finite group and
$\pi_e(G)$ be the set of element orders in $G$. If $\pi_e(G) \cap
\{2\}=\emptyset$; or $\pi_e(G) \cap \{3, 4\}=\emptyset$; or
$\pi_e(G) \cap \{3,5\}=\emptyset$, then $G$ is solvable. Moreover,
using the intersection with $\pi_e(G)$ being empty set to judge $G$
is solvable or not, only the above three cases.
\end{abstract}

\baselineskip .65 true cm

\section{Introduction}
Let $G$ be a finite group. We have two basic sets: $|G|$ and
$\pi_e(G)$. There are many famous works about $|G|$ in the history
of group theory. The set $\pi_e(G)$ was studied first by author in
\cite{S-A}. The main results in \cite{S-A} are:

\begin{lemma} \label{l1}
    Let $G$ be a finite group. If $\pi_e(G)=\{1,2,3,5\}$, then $G \cong A_5$.
\end{lemma}

One can easily get the following conclusion from \cite{S-J}:

\begin{lemma} \label{l2}
    Let $G$ be a finite group with $\pi_e(G)=\{2, $ the factors of $(2^n-1)$ and $(2^n+1), n \ge 2\}$. Then $G \cong L_2(2^n)$.
\end{lemma}
For the simple group $Sz(2^{2m+1})$, we have the following
result(\cite{S-P}):

\begin{lemma}\label{l3}
    Let $G$ be a finite group. If $\pi_e(G)=\{2,4, $ the factors of $(2^{2m+1}-1), (2^{2m+1}-2^{m+1}+1)$ and $(2^{2m+1}+2^{m+1}+1), m \ge 1\}$, then $G \cong Sz(2^{2m+1})$.
\end{lemma}
We will use the above three lemmas in the following discussion.

Compare with the study of $|G|$, we can also ask similar question
about $\pi_e(G)$:

In \cite{S-S}, one can find the definition of $h$ function: for a set $S$ of element orders of a finite group, $h(S)$ is defined to be the number of non-isomorphic groups $G$ with $\pi_e(G)=S$. There are many study about the $h$ function with $h(S)=1$, which means that the group $G$ can be characterized by $\pi_e(G)$. The recent study can be find in \cite{V-G}.

Same as CLT-group (\cite{H}), we study COE-group in \cite{S7}.

As the solvability can be decided by the order of a group
(\cite{H-P}), we give condition about solvabilty by $\pi_e(G)$. This
can be seen as a supplement of the following theorem in \cite{D}:

\begin{theorem}
    Let $G$ be a finite group and $\pi_e(G)$ be the set of element orders. Let $|\pi(G)|$ be the number of prime in $\pi_e(G)$ and $|\chi(G)|$ be the number of composite number in $\pi_e(G)$. Then $|\pi(G)| \le |\chi(G)|+3$, and $G$ is simple if $|\pi(G)|=|\chi(G)|+3$.
\end{theorem}

\begin{definition}
    Let $G$ be a finite group and $\pi_e(G)$ be the set of element orders of $G$. A set $T$ is called IES if $T \cap \pi_e(G)=\emptyset$ implies $G$ is solvable.
\end{definition}

If an IES-set $T$ with $|T|=1$, then by Feit-Thompson theorem, we get $T=\{2\}$. For the other set $T$, there are examples of non-solvable groups. So we get

For an IES-set, $|T|=1$ if and only if $T=\{2\}$.

We can consider the IES-set $T$ with $|T|=2$. If $T=\{2, *\}$, where
$*$ is an integer $>2$, we can get $G$ is solvable and this is
trivial. We need to consider the case that $T=\{3, **\}$, where
$**>3$.

J. G. Thompson classified all the minimal simple groups in \cite{Th}:

\begin{lemma} \label{l4}
        The minimal simple groups are:

    (1) $L_2(p)$, $p>3$, $5\nmid (p^2-1)$, where $p$ is a prime;

    (2) $L_2(2^p)$, where $p$ is a prime;

    (3) $L_2(3^p)$, where $p$ is a odd prime;

    (4) $L_3(3)$;

    (5) $Sz(2^{2m+1})$, where $2m+1$ is a prime.
\end{lemma}

Notice that in the above lemma, the simple groups in (1)-(4)  all
have an element of order $3$. For the simple group $Sz(2^{2m+1})$,
since $\pi_e(G)=\{2,4, $ the factors of $(2^{2m+1}-1),
(2^{2m+1}-2^{m+1}+1)$ and $(2^{2m+1}+2^{m+1}+1), m \ge 1\}$, we get
that $3 \not \in \pi_e(Sz(2^{2m+1})))$, but $4, 5 \in
\pi_e(Sz(2^{2m+1}))$($5 \mid
(2^{2m+1}-2^{m+1}+1)(2^{2m+1}+2^{m+1}+1)$ for $5\mid (2^{4m+2}+1)$).
Hence if $|T|=2$, then $T=\{3,4\}$ and $\{3,5\}$ are IES-sets. We
claim that there are no other IES-set $T$ with $|T|=2$.

(a) Let $T=\{3, x\}$, where $x >5$. Notice $\pi_{e}(Sz(2^3))=\{1,2,4,5,7,13\}$ and $\pi_e(Sz(2^5))=\{1,2,4,5,25,31,41\}$. We get that the common element orders in the above two minimal simple groups are $\{1,2,4,5\}$. So we can get a conterexample for any $x>5$.

(b) Let $T=\{4,y\}$, where $y>4$. Notice $\pi_e(A_5)=\{1,2,3,5\}$.
To exclude this case, we know $y=5$. Since
$\pi_e(L_2(2^3))=\{1,2,3,7,9\}$, thus such $y$ such that $T$ is an
IES-set does not exist.

(c) Let $T=\{5, z\}$, where $z>5$. Since
$\pi_e(L_2(2^3))=\{1,2,3,7,9\}$ and
$\pi_e(L_2(2^5))=\{1,2,3,11,31,33\}$, we get that such $T$ does not
exists by the same reason.

For the other $T$ with $|T|=2$, we can find a counterexample from $\pi_e(A_5)=\{1,2,3,5\}$.

Next we consider the case of $|T|=3$. Let $T=\{n_1, n_2, n_3\}$.
From the above discussion, $n_1, n_2, n_3$ are odd. We can assume
that $n_1<n_2<n_3$.

Suppose $n_1=3$. If $n_2$ or $n_3 \in \{4,5\}$, it is a trivial consequence of the above discussion.

So we can assume that $n_2>5$ and $m_3>6$. Since
$\pi_e(Sz(2^3))=\{1,2,4,5,7,13\}$,
$\pi_e(Sz(2^5))=\{1,2,4,5,25,31,41\}$ and
$\pi_e(Sz(2^7))=\{1,2,4,5,29,113,127,145\}$, we can get an
counterexample for any $n_2, n_3$. Hence such $T$ exists.

Suppose $n_1=4$. We can assume that $n_2>4$ and $n_3>5$. Since
$\pi_e(L_2(2^3))=\{1,2,3,7,9\}$,
$\pi_e(L_2(2^5))=\{1,2,3,11,31,33\}$ and
$\pi_e(L_2(2^7))=\{1,2,3,43,127,129\}$, we can get an counterexample
for any $n_2, n_3$. Hence such IES-set does not exist. We can also
get the same conclusion for $n_1=5$.

For $n_1>5$, $\pi_e(A_5)=\{1,2,3,5\}$ will give us an counterexample.

Therefore, the nontrivial IES-set $T$ with $|T|=3$ does not exist.

Finally, we consider the case of IES-set $T$ with $|T|>3$. We claim
that no such IES-set exists. To do this, we give the following
lemma.

\begin{lemma}
    \label{l5}
    Let $m, n$ be two integer with $d=(m,n)$. Then $(2^m-1, 2^n-1)=2^d-1$.
\end{lemma}
\begin{proof}
    Suppose that $m=nq+r$, where $0\le r<n$. Then $2^{m}-1=2^{nq+r}-1=2^r(2^{nq}-1)+2^r-1$$\equiv$$2^r-1(\mod{2^n-1})$. So $(2^m-1, 2^n-1)=(2^r-1, 2^n-1)$. In this way, we can get
     $$(2^m-1, 2^n-1)=2^d-1.$$
\end{proof}

This lemma is a direct consequence of Theorem 1 of Section 7.4 in \cite{K}.

\begin{corollary}
    Let $p, q$ be two different primes. Then $\pi_e(L_2(2^p)) \cap \pi_e(L_2(2^q))=\{1,2,3\}$.
\end{corollary}
\begin{proof}
It is well-known that $|L_2(2^p)|=2^p(2^{2p}-1)$,
$|L_2(2^q)|=2^q(2^{2q}-1)$, and $(2p, 2q)=2$. By Lemma \ref{l5},
there is no common element order in $L_2(2^p)$ and $L_2(2^q)$ except
$1,2,3$.
\end{proof}

\begin{corollary}
    Let $p,q$ are two different odd primes. Then $\pi_e(Sz(2^p)) \cap \pi_e(Sz(2^q))=\{1,2,4,5\}$.
\end{corollary}
\begin{proof}
    We first prove that $(2^{2p}+1, 2^{2q}+1)=5$ for any different odd prime $p, q$.

Clearly, $(2^{2p}+1, 2^{2q}+1) \mid (2^{4p}-1, 2^{4q}-1)$. By Lemma \ref{l5}, $(2^{4p}-1, 2^{4q}-1)=2^4-1=3\times 5$. Since $3 \nmid 2^{2p}+1$, $(2^{2p}+1, 2^{2q}+1)=5$.

In the same way, we can get that $(2^{2p}+1, 2^q-1)=(2^{2q}+1, 2^p-1)=1$.

It is easy to get that $\pi_e(Sz(2^p)) \cap
\pi_e(Sz(2^q))=\{1,2,4,5\}$.
\end{proof}

Now we continue the discussion of IES-set $T$ with $|T|= s>3$. Let
$T=\{n_1, n_2, \cdots, n_s\}$, where $n_1 <n_2 < \cdots< n_s$ and
$s>3$.

As above, the case that $n_1=2$, or $\{n_1, n_2\}=\{3,4\}$, or $\{n_1, n_2\}=\{3,5\}$ is trivial. If $n_1>5$, $A_5$ will be a counterexample. So we need to consider the following cases:

(d) Let $T=\{3,n_2,\cdots, n_s\}$, where $n_2>5$. Notice that
$\pi_e(Sz(2^p))\cap \pi_e(Sz(2^q))=\{1,2,4,5\}$. We get that  there
exists $k\in \pi_e(Sz(2^p))-\{1,2,4,5\}$ and $k \not \in
\pi_e(Sz(2^q))$. Hence for any $s=|T|$, there are enough large
primes $p$ to provide counterexample. Thus no such $T$ exists.

(e) Let $T=\{4,n_2,\cdots, n_s\}$, where $n_2>4$. Notice
$\pi_e(L_2(2^p)) \cap \pi_e(L_2(2^q))=\{1,2,3\}$. We get that there
exists $k\in \pi_e(L_2(2^p))-\{1,2,3\}$ and $k \not \in
\pi_e(L_2(2^q)), q \neq p$. Hence for any $s=|T|$, there are enough
large primes $p$ to provide counterexample. Thus no such $T$ exists.

Hence we get:
\begin{theorem}
    Let $G$ be a finite group and $\pi_e(G)$ be the set of element orders of $G$.
    If $2 \not \in \pi_e(G)$, $\pi_e(G) \cap \{3,4\}=\emptyset$, or $\pi_e(G) \cap \{3,5\}=\emptyset$, then $G$ is solvable. Furthermore, $T$ is an IES-set if and only if $T=\{2\}$, $\{3,4\}$ or $\{3,5\}$.
\end{theorem}

{\bf Acknowledgements}  The author would like to thank Prof. Ming
Luo for his help in number theory.

\baselineskip -.2 true cm
\bibliographystyle{plain}

\begin{thebibliography}{10}

\bibitem{S-A}
Wujie Shi, A Characterization of $A_5$(in Chinese), J. Southwest
China Normal Univ. (Natural Soc.), 3(1986), 11-14.


\bibitem{S-J}
Wujie Shi, The characterization of $J_1$ and $PSL_2(2^n)$(in
Chinese), Advance in Math. (China), 16:4(1987), 397-401.

\bibitem{S-P}
Wujie Shi, A characterization of Suzuki's simple groups, Proc. Amer.
Math. Soc., 114:3(1992), 589-591.

\bibitem{S-S}
Wujie Shi, The finite groups with given set of element orders(in
Chinese), Chinese Science Bulletin, 42:16(1997), 1703-1706.

\bibitem{V-G}
A.V. Vasil¡¯ev and M.A. Grechkoseeva, Recognition by spectrum for
simple classical groups in characteristic 2, Siberian Math. J.,
56:6(2015), 1009-1018.

\bibitem{H}
J.F. Humphregs, On groups satisfying the convers of Lagrange¡¯s
theorem, Proc. Camb. Phil. Soc., 75(1974), 25-32.

\bibitem{S7}
Wujie Shi, Finite groups defined by the sets of their element
orders, J. Southwest China Normal Univ. (Natural Soc.), 22:5(1997),
481-486.


\bibitem{H-P}
Junhua He and Wei Pu, On the number $n$ which makes any finite
groups are solvable with order prime to $n$(in Chinese), J.
Southwest China Normal Univ. (Natural Soc.), 24:6(1999), 612-614.


\bibitem{D}
H. Deng and W. Shi, A simplicity criterion for finite groups, J.
Algebra, 191:1(1997), 371-381.

\bibitem{Th}
J.G. Thompson, Nonsolvable finite groups all of whose local
subgroups are solvable I, Bull. Amer. Math. Soc. 74(1968), 383-437.

\bibitem{K}
Zhao Ke and Qi Sun, The Lecture of the Number Theory(in Chinese),
The Second Edition, Part two, page 14.
\end{thebibliography}

\end{document}